\documentclass[11pt]{article}
\usepackage{mathrsfs}
\usepackage{amsmath}
\usepackage{amsthm}
\usepackage{graphicx}
\usepackage{amsfonts}
\usepackage{amssymb}
\usepackage{enumerate}
\usepackage{float}
\usepackage{epsfig}
\usepackage{color}
\newtheorem{thm}{Theorem}[section]
\newtheorem{Lemma}[thm]{Lemma}
\newtheorem{cor}[thm]{Corollary}

\makeatletter
\long\def\@makecaption#1#2{%
 \vskip\abovecaptionskip
  \sbox\@tempboxa{{#1.}\quad #2}%
   \ifdim \wd\@tempboxa >\hsize
    { #1.}\quad #2\par
     \else
  \global \@minipagefalse
   \hb@xt@\hsize{\hfil\box\@tempboxa\hfil}%
   \fi
   \vskip\belowcaptionskip}
\makeatother

\usepackage{latexsym, bm}
\title{\vspace{0cm}{Matching preclusion for vertex-transitive networks\thanks{This
 work is supported by NSFC (nos. 11371180 and 11401279), SRFDP (no. 20130211120008) and the Fundamental Research Funds for the Central
Universities (no. lzujbky-2014-21).}}}

\author{Qiuli Li, Jinghua He and Heping Zhang\footnote{The
corresponding author.} \\
\small{School of Mathematics and Statistics, Lanzhou University,
Lanzhou, Gansu 730000, P. R. China}\\
\small{E-mail addresses:  qlli@lzu.edu.cn, hejh@lzu.edu.cn and
zhanghp@lzu.edu.cn}}

\date{\today}

\begin{document}
\maketitle
\begin{abstract}
In interconnection networks, matching preclusion is a measure of
robustness when there is a link failure. Let $G$ be a graph of even
order. The matching preclusion number $mp(G)$ is defined as the
minimum number of edges whose deletion results in a subgraph without
perfect matchings. Many interconnection networks are super matched,
that is, their optimal matching preclusion sets are precisely those
induced by a single vertex. In this paper, we obtain general results
of vertex-transitive graphs including many known networks. A
$k$-regular connected vertex-transitive graph has matching
preclusion number $k$ and is super matched except for six classes of
graphs. From this many previous results can be directly obtained and
matching preclusion for some other
 networks, such as folded $k$-cubes, Hamming graphs and halved $k$-cubes, are derived.



\end{abstract}
\section{Introduction}

 A \emph{network} (or \emph{graph}) is a collection of points or nodes, called \emph{vertices},
 and a collection
 of links, called \emph{edges}, each connecting two nodes.
 The number of vertices of a graph $G$ is its \emph{order}, written $|G|$; its number of edges
 is denoted by $||G||$. We use $V(G)$ and $E(G)$ denote the vertex-set and edge-set of $G$
  respectively. Throughout this article, all graphs are assumed to be connected and
 of even order.
 The matching preclusion, viewed as a measure of the robustness
 of graphs, of many networks has been investigated. By summarizing
 these results, we can see that almost all the networks considered are vertex-transitive
 and surprisingly,
 their matching preclusion almost act in the same way. A natural
 question arises: What does the matching preclusion of
 vertex-transitive graphs act? More precisely, can we obtain a unified property on the matching
 preclusion of vertex-transitive graphs?

 A \emph{perfect matching} in a graph is a set of edges such that every vertex is incident
 with exactly one edge in this set. For $S\subseteq E(G)$, if $G-S$ has no perfect matchings,
 where $G-S$ denotes the subgraph of $G$ by deleting $S$ from it, then we call $S$ a
 \emph{matching preclusion set}. The \emph{matching preclusion number}
of a graph $G$, denoted by $mp(G)$, is the minimum cardinality among
all matching preclusion sets. Correspondingly, the matching
preclusion set attaining the matching preclusion number is called an
\emph{optimal matching preclusion set} (or in short, \emph{optimal
solution}). The concept of matching preclusion was introduced by
Brigham et al. for ``measuring the robustness of a communications
network graph which is a  model  for the distributed algorithm that
require each node of it to be matched with a neighboring partner
node ''\cite{Brigham2005}.

Until now, the matching preclusion numbers of lots of networks
(graphs) have been computed, such as Petersen graph, hypercube,
complete graphs and complete bipartite graphs \cite{Brigham2005},
Cayley graphs generalized by transpositions and $(n,k)$-star graphs
\cite{Cheng2007}, augmented cubes \cite{Cheng101},
$(n,k)$-buddle-sort graphs \cite{Cheng10}, tori and related
Cartesian products \cite{Cheng20121},  burnt pancake graphs
\cite{Hu13}, balanced hypercubes \cite{lv2012}, restricted HL-graphs
and recursive circulant $G(2^{m}, 4)$ \cite{Park2008}, and $k$-ary
$n$-cubes \cite{Wang2010}. Their optimal solutions have been also
classified.

By deleting the edges incident with a given vertex in a graph, the
resulting subgraph has no perfect matchings. Hence the matching
preclusion number is bounded by the minimum degree.

\begin{thm}[\cite{Cheng2007}]\label{upperbound0}
Let $G$ be a graph of even order. Then $mp(G) \leq \delta(G)$, where
$\delta(G)$ is the minimum degree of $G$.
\end{thm}

In a network, a vertex with a special matching vertex after edge
failure any time implies that tasks running on a fault vertex can be
shifted onto its matching vertex. Thus under this fault assumption,
larger $mp(G)$ signifies higher fault tolerance. Fortunately,
matching preclusion numbers of many regular interconnection networks
of degree $k$ attained the maximum value, the minimum degree $k$.
Moreover, the optimal solutions are precisely those induced by a
single vertex except the ones with small order. Formally, we call
the optimal solution incident with a single vertex a \emph{trivial
optimal solution} (\emph{non-trivial optimal solution} otherwise)
and the graphs with all optimal solutions trivial \emph{super
matched}. Generally, in the event of a random link failure, it is
very unlikely that all of the links incident to a single vertex fail
simultaneously. From this point of view, that a graph is super
matched implies that it has higher fault tolerance.

Recalling that the networks whose matching preclusion have been
considered, we can see that many of them are  vertex-transitive
graphs.
A graph $H$ is called \emph{vertex-transitive} if for any two
vertices $x,y$ in $V(H)$, there exists an automorphism $\varphi$ of
$H$ such that $\varphi(x)=y$. From the known results, we can see
that almost all of them are super matched. 
Fortunately, we obtain that almost all vertex-transitive graphs have
such properties, too. Precisely, we get the following result, in
which, $Z_{4n}(1,4n,2n)$ stands for the Cayley graph on $Z_{4n}$,
the additive group modulo $4n$, with the generating set
$S=\{1,4n-1,2n\}$. $Z_{4n+2}(2,4n,2n+1)$ and
$Z_{4n+2}(1,4n+1,2n,2n+2)$ are defined similarly.

\begin{thm}\label{mainresult}
A $k$-regular connected vertex-transitive graph $G$ of even order is
super matched if and only if it doesn't contain cliques of size $k$
when $k$ is odd and $k\leq |G|-2$ or it is not isomorphic to a cycle
of length at least six or $Z_{4n}(1,4n-1,2n)$ or
$Z_{4n+2}(2,4n,2n+1)$ or $Z_{4n+2}(1,4n+1,2n,2n+2)$ or the Petersen
graph.
\end{thm}

 This article is organized as follows.  In Section 2, we will
 analyse some structural properties of vertex-transitive graphs.
 In Section 3, we present the proof of Theorem \ref{mainresult}.
 In Section 4, we make a conclusion and several applications to obtain the
 matching preclusion of some networks.

\section{Preliminaries}

In this section, we shall present several results that will be used
later. An edge set $S\subseteq E(G)$ is called an \emph{edge-cut} if
there exists a set $X\subseteq V(G)$ such that $S$ is the set of
edges between $X$ and $V(G)\setminus
 X$. The\emph{ edge-connectivity} $\lambda(G)$ of $G$ is the minimum
 cardinality over all edge-cuts of $G$. Mader proved the following
 result.
\begin{Lemma}[\cite{Mader71}]\label{Mader71}
If $G$ is a $k$-regular connected  vertex-transitive graph, then
$\lambda(G)=k$.
\end{Lemma}

 The following lemma makes a step further by characterizing the minimum edge-cuts of
vertex-transitive graphs, where a clique of a graph $G$ is a subset
of its vertices such that every two vertices in the subset are
connected by an edge.

\begin{thm}[\cite{LovaszandPlummer86}, Lemma 5.5.26]\label{kedgeconnected}
Let $G$ be a $k$-regular connected vertex-transitive graph. Then
$\lambda(G)=k$ and either

\noindent (i) every minimum edge-cut of $G$ is the star of a vertex,
or

\noindent (ii) $G$ arises from a (not necessarily simple) vertex-
and edge-transitive $k$-regular graph $G_{0}$ by a $k$-clique
(cliques of size $k$) insertion at each vertex of $G_{0}$. Moreover,
every minimum edge-cut of $G$ is the star of a vertex of $G$ or a
minimum edge-cut of $G_{0}$.
\end{thm}

The following corollary that will be used in Section 3 follows
immediately. An edge-cut is called \emph{trivial} if it isolates a
vertex and \emph{non-trivial} otherwise.

\begin{cor} \label{minedgecut}
For a $k$-regular connected vertex-transitive graph $G$, every
$k$-edge-cut (an edge-cut of size $k$) is either trivial or the
deletion of it results in two components, and each component's
vertices are partitioned into several $k$-cliques.
\end{cor}

For  a $k$-regular graph $G$, if every minimum edge-cut of it is
trivial, then we say it is
 \emph{super-edge-connected} (or simply
\emph{super-$\lambda$}). For vertex-transitive graphs, J. Meng has
presented a characterization with respect to the cliques.

\begin{thm}[\cite{Meng2003}]\label{superlambda}
Let $G$ be a $k$-regular connected vertex-transitive graph which is neither a complete graph nor a cycle. Then $G$ is super-$\lambda$
if and only if it does not contain $k$-cliques.
\end{thm}

Theorem \ref{superlambda} is used to characterize the structure of
3-regular connected non-bipartite vertex-transitive graphs with
respect to the length of the minimum odd cycles. As we will see,
minimum odd cycles play a crucial role in the following proof in
this section. Here we make a convention that is suitable throughout
this paper. For a cycle drawn on the plane without crossings, let
$a,b\in V(C)$, denote $P_{ab}$ by the path of $C$ from $a$ to $b$
along a clockwise direction.
A cycle $C$ is called a \emph{minimum odd cycle} if $|C|$ is odd and
it is the minimum among lengths of all odd cycles. For a minimum odd
cycle, we usually say it is\emph{ minimum}. The following two
results (Lemmas \ref{adjacenttoatmosttwo} and \ref{k3g4}) will be
used to prove Lemma \ref{structure}, which play an important role in
the proof of Theorem \ref{mainresult} in the next section.

\begin{Lemma}\label{adjacenttoatmosttwo}
Let $G$ be  a non-bipartite connected vertex-transitive graph of
even order and $C$ be a minimum odd cycle in $G$ with $|C|\geq
\frac{|G|}{2}$. If $G$ is not isomorphic to $K_{4}$ or $K_{6}$, then
any vertex in $V(G)\setminus V(C)$ is adjacent to at most two
vertices in $V(C)$. If in addition, a vertex $v\in V(G)\setminus
V(C)$ is adjacent to two vertices $u$ and $w$ in $V(C)$, then either
$P_{uw}$ or $P_{wu}$ in $C$ is of length 2 and further $G$ has a
quadrangle containing $v$, $u$ and $w$.
\end{Lemma}
\begin{proof} We first draw $C$ on the plane without crossings. Suppose by
the contrary that  there exists a vertex $v\in V(G)\setminus V(C)$
with at least three neighbors in $V(C)$. Let $a$, $b$ and $c$ be
three neighbors of $v$ in $V(C)$ and they are placed in $C$ in a
successive order along the clockwise direction. Since
$||P_{ab}||+||P_{bc}||+||P_{ca}||=||C||$ is odd, at least one of
$||P_{ab}||$, $||P_{bc}||$ and $||P_{ca}||$ is odd, we may assume
that $||P_{ab}||$. Hence $||P_{bc}||+||P_{ca}||$ is even. If
$||P_{bc}||+||P_{ca}||=2$, then $vac$ is an odd cycle of length
three, so is $C$ since $C$ is minimum. Further, $|G|\leq 2||C||=6$,
we can easily check that $G$ is isomorphic to $K_{4}$ or $K_{6}$,
which contradicts the hypothesis. If $||P_{bc}||+||P_{ca}||\geq 4$,
then $E(P_{ab})\cup \{va, vb\}$ induces an odd cycle of smaller
length than $C$, which contradicts that $C$ is minimum.

By the above arguments, any vertex in $V(G)\setminus V(C)$ is
adjacent to at most two vertices in $V(C)$. If $v\in V(G)\setminus
V(C)$ is adjacent to exactly two vertices $u$ and $w$ in $V(C)$,
then by $||P_{uw}||+||P_{wu}||=||C||$ is odd, either $||P_{uw}||$ or
$||P_{wu}||$ is even. We may assume that $||P_{uw}||$ is even.
Further, the edges of $P_{wu}$ and  $\{vu,vw\}$ form an odd cycle of
length at least $||C||$. Hence $||P_{uw}||=2$ and $E(P_{uw})\cup
\{vu, vw\}$ induces a quadrangle containing $v$, $u$ and $w$.
\end{proof}

For a graph $G$, let the girth (odd girth) of $G$, denoted by $g(G)$
($g_{o}(G)$), be the length of a shortest cycle (odd cycle) in $G$.
We say two quadrangles in $G$ \emph{adjacent} if they share vertices
or edges.

\begin{Lemma}\label{k3g4}
Let $G$ be a 3-regular connected non-bipartite vertex-transitive
graph of girth 4. If $G$ has adjacent quadrangles, then it is
isomorphic to $Z_{4n}(1, 4n-1, 2n)$ or $Z_{4n+2}(2, 4n, 2n+1)$ with
$n\geq 2$. Otherwise $g_{o}(G)\leq \frac{|G|}{2}$.
\end{Lemma}

\begin{proof} By Theorem \ref{Mader71}, $G$ is 3-edge-connected.
Further, since $G$ is of girth 4, that is, $G$ does not contains
3-cliques (triangles), every 3-edge-cut is trivial by Theorem
\ref{superlambda}.


If $G$ has adjacent quadrangles $q_{1}$ and $q_{2}$, then by the
3-regularity of $G$, $q_{1}$ and $q_{2}$ should share at least one
edge. If $q_{1}$ and $q_{2}$ share exactly three edges, then we
obtain multiple edges, contradicting that $G$ is simple; If $q_{1}$
and $q_{2}$ share exactly two edges, let $H$ be the subgraph of $G$
induced by the edges of $q_{1}$ and $q_{2}$, then $H$ has exactly
three vertices of degree 2 and the others of degree 3. Since every
3-edge-cut of $G$ is trivial, we obtain that the three degree-2
vertices in $H$ are adjacent to a common neighbor, that is, $G$ is
$K_{3,3}$, which contradicts that $G$ is non-bipartite. By the above
arguments, $q_{1}$ and $q_{2}$ should share exactly one edge, that
is, $G$ contains $K_{2}\times P_{3}$ as a subgraph, where $`\times$'
means the Cartesian product of graphs and $P_{3}$ denotes a path
$P_{m}$ with $m=3$ vertices. Let $K_{2}\times P_{m}$ with $m\geq 3$
be a subgraph of $G$ with $m$ maximum (see Figure \ref{34} (left)).
For the four vertices $x_{1}, y_{1}, x_{m}$ and  $y_{m}$, if any two
of them are adjacent to each other, then the remaining two of them
are adjacent by the 3-connectivity of $G$, this implies that $G$ is
isomorphic to $Z_{2m}(1, 2m-1, m)$ when $m$ is even or $Z_{2m}(2,
2m-2, m)$ when $m$ is odd with $m\geq 3$ (see Figure \ref{34} (the
middle and right ones)). If no two of them are adjacent to each
other, then we may suppose that $x_{m}$ is adjacent to $x_{m+1}$ and
$y_{m}$ is adjacent to $y_{m+1}$. By the vertex-transitivity of $G$,
we know that there are at least two adjacent quadrangles containing
$y_{m}$. Then by a simple check, we can obtain that $x_{m+1}$ and
$y_{m+1}$ should be adjacent, then we find a subgraph $K_{2}\times
P_{m+1}$ in $G$, which contradicts our selection.



\begin{figure}[!htbp]
\begin{center}
\includegraphics[totalheight=2.8cm]{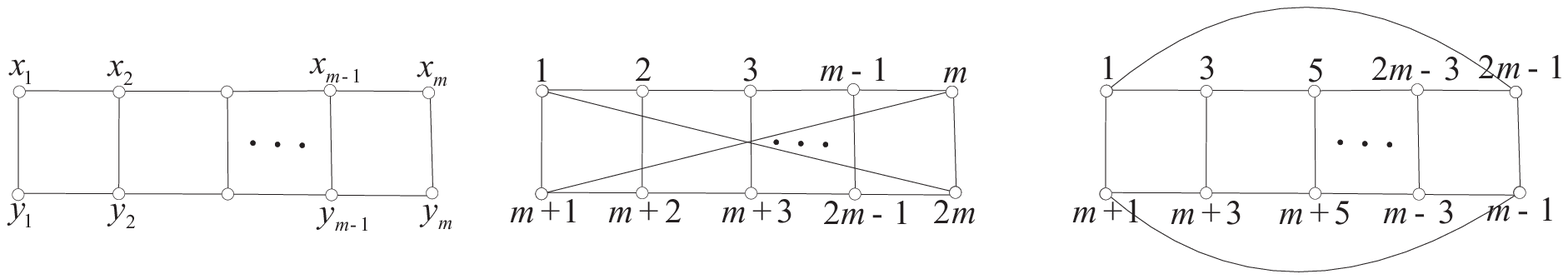}
 \caption{\label{34}}
\end{center}
\end{figure}

If $G$ does not contain adjacent quadrangles, then every vertex lies
in exactly one quadrangle. Suppose to the contrary that
$g_{o}(G)\geq \frac{|G|}{2}+1$. Let $C$ be a minimum odd cycle with
length $g_{o}(G)$. We shall obtain a contradiction by proving that
there are at least $|C|$ vertices in $V(G)\setminus V(C)$. Since $G$
is not $K_{4}$ or $K_{6}$ (by the hypothesis that $G$ is of girth 4)
and $C$ is minimum, by Lemma \ref{adjacenttoatmosttwo}, any vertex
in $V(G)\setminus V(C)$ is adjacent to at most two vertices in
$V(C)$.


Since every vertex lies in a quadrangle and $C$ is minimum, each
quadrangle containing vertices in $V(C)$ can only contain one or two
edges in $C$. Denote $G'$ by the subgraph of $G$ induced by the
edges in $C$ and the quadrangles having a non-empty intersection
with $C$ (see Figure \ref{G1} (left)). For any vertex $a\in V(C)$,
if it is of degree 2 in $G'$, then we claim that all its neighbors
in $G$ cannot lie entirely in $V(G')$. Suppose not. Since $C$ is
minimum, $C$ has no chords, we obtain the two structures shown in
Figure \ref{G1} (the middle and right ones). For each case, by Lemma
\ref{adjacenttoatmosttwo}, $v$, $a$ and $b$ lie in a quadrangle.
Then we obtain two adjacent quadrangles, a contradiction.

\begin{figure}[!htbp]
\begin{center}
\includegraphics[totalheight=5cm]{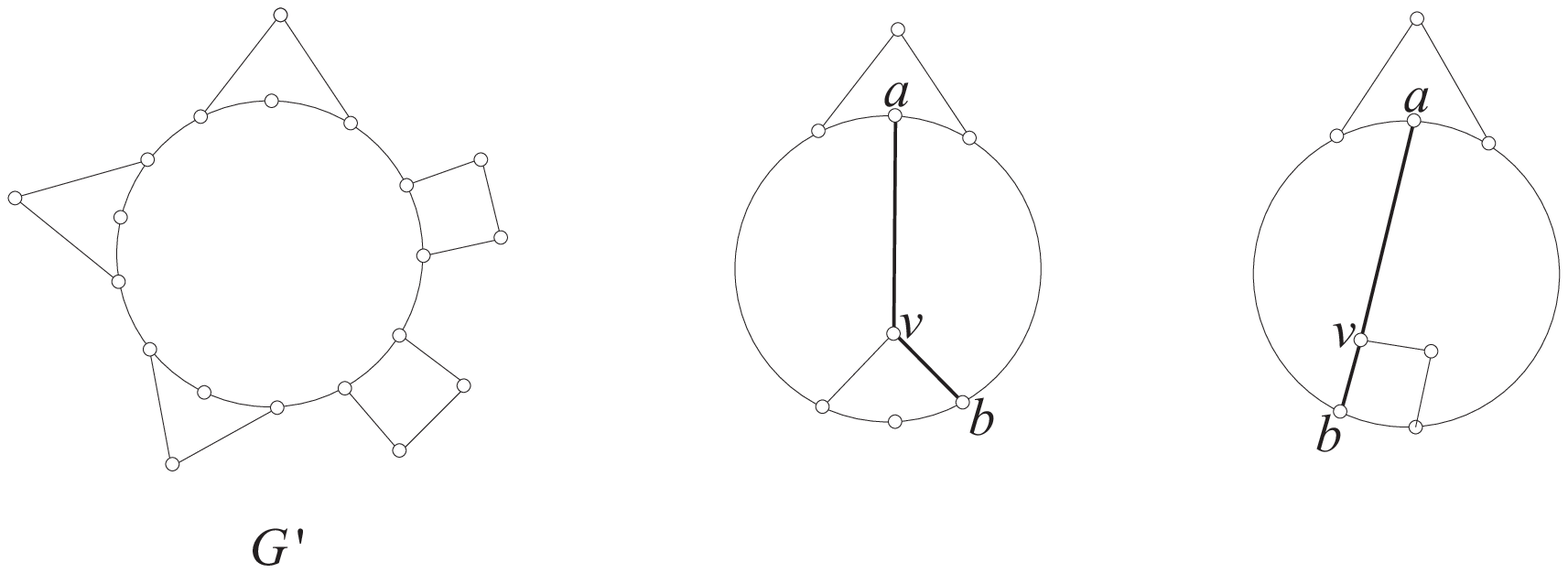}
 \caption{\label{G1}}
\end{center}
\end{figure}

In $G'$, for all the degree-2 vertices in $V(C)$, we collect the
edges incident to them in $E(G)\setminus E(G')$, denoted by $F$ (the
bold edges in Figure \ref{G2} (left)). We claim that $F$ is a
matching. If not, then we obtain a vertex $v$ adjacent to two
vertices $a$ and $b$ in $V(C)$. By Lemma \ref{adjacenttoatmosttwo},
$v$, $a$ and $b$ lie in a quadrangle. Then we see that both $a$ and
$b$ lie in two quadrangles, a contradiction.

Denote $G''$ by the subgraph induced by $E(G')\cup F$ in $G$, and
denote the set of degree-1 vertices in $G''$ be $A$. It is obvious
that $|A|=|F|$. We are going to show that there are at least $|A|$
vertices in $V(G)\setminus V(G'')$. If so, then by a simple
computation (for the vertices in $V(C)$ lying in the quadrangles
sharing exactly one edge with $E(C)$, there is the same number of
vertices in the quadrangles in $V(G)\setminus V(C)$; for the
vertices in $V(C)$ lying in the quadrangles sharing exactly two
edges with $E(C)$, by counting the vertices in quadrangles not in
$V(C)$, the vertices contained in $F$ not in $V(C)$ and the vertices
in $V(G)\setminus V(G'')$, we obtain the same conclusion too), there
are at least $|C|$ vertices in $V(G)\setminus V(C)$, which
contradicts that $C$ is of length at least $\frac{|G|}{2}+1$.

\begin{figure}[!htbp]
\begin{center}
\includegraphics[totalheight=5cm]{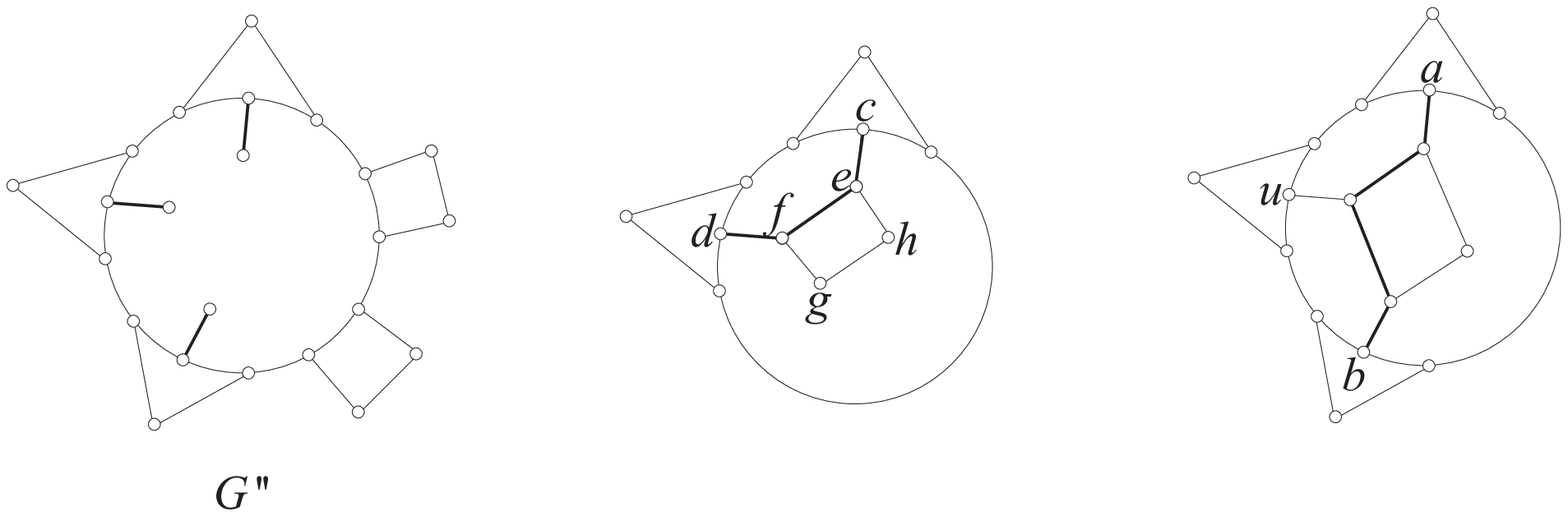}
 \caption{\label{G2}}
\end{center}
\end{figure}

We are left to prove that there are at least $|A|$ vertices in
$V(G)\setminus V(G'')$. For any vertex $v\in A$, let the quadrangle
containing it be $Q_{v}$. If $Q_{v}$ contains three vertices in
$V(G)\setminus V(G'')$, then any vertex $u$ in $A$ other than $v$
cannot be adjacent to any vertex in this quadrangle (otherwise,
$Q_{u}$ and $Q_{v}$ are adjacent, a contradiction), and we count
three for $V(G)\setminus V(G'')$; If the quadrangle containing it
contains exactly two vertices $g$ and $h$ in $V(G)\setminus V(G'')$
(see the middle one in Figure \ref{G2}), similarly, any vertex in
$A$ other than $e$ and $f$ can not be adjacent to $g$ and $h$ (if
there does exist such a vertex $w$, then by considering the
quadrangle containing $w$, we obtain two adjacent quadrangles, a
contradiction), then we count two for these two vertices $e$ and $f$
in $A$. Moreover, since $C$ is a minimum odd cycle, by a similar
argument as above, we can deduce that $||P_{cd}||=3$ or
$||P_{dc}||=3$; If the quadrangle containing it contains at most one
vertex in $V(G)\setminus V(G'')$ (see the right one in Figure
\ref{G2}), then we can deduce that $||P_{bu}||=||P_{ua}||=3$. By
substituting the edges in $P_{ba}$ in $C$ with the bold edges, we
obtain an odd cycle of length smaller than $C$, which contradicts
that $C$ is minimum. Therefore, there are at least $|A|$ vertices in
$V(G)\setminus V(G'')$. This completes the proof.
\end{proof}

Now everything is ready to prove the following key lemma which
characterizes the vertex-transitive graphs with provided structure.

\begin{Lemma}\label{structure}
Let $G$ be a $k$-regular connected non-bipartite vertex-transitive
graph. Let $S\subseteq V(G)$ with $|S|=|\overline{S}|+2$, where
$\overline{S}=V(G)\setminus S$, and $\overline{S}$ is an independent
set of $G$. Then $G$ is isomorphic to $Z_{4n+2}(1,4n+1,2n, 2n+2)$ or
$Z_{4n}(1,4n-1,2n)$ or $Z_{4n+2}(2,4n,2n+1)$ or the Petersen graph.
\end{Lemma}

\begin{proof} Obviously, $\overline{S}\neq \emptyset$.
Let $C$ be a minimum odd cycle and $l=|C|$.

{\bf {Claim }1.} $l=g_{o}(G)\geq \frac{|G|}{2}$.

Suppose that there are $n_{l}$ minimum odd cycles in $G$ and each
vertex is contained in $v_{l}$ minimum odd cycles. Since
$G[\overline{S}]$ is empty, each minimum odd cycle contains at most
$\frac{l-1}{2}$ vertices in $\overline{S}$ and at least
$\frac{l+1}{2}$ vertices in $S$. By counting the number (repeated by
re-number calculation) of minimum odd cycles passing through the
vertices in $\overline{S}$, we have

$$v_{l}\times |\overline{S}|\leq \frac{l-1}{2}n_{l}.$$

Similarly for $S$, we have

$$v_{l}\times |S|\geq \frac{l+1}{2}n_{l}.$$

Combining the above two inequalities with $|S|=|\overline{S}|+2$, we
have $v_{l}\geq \frac{n_{l}}{2}$. Consequently, by counting the
number of times (repeated by re-number calculation) that the minimum
odd cycles passing through all the vertices in $G$, we obtain that
$n_{l}\times l =v_{l}\times |G|\geq \frac{n_{l}}{2}\times |G|$,
which implies $l\geq \frac{|G|}{2}$. So Claim 1 holds.

Now we show that only vertex-transitive graphs with small $k$
satisfy the conditions in the lemma.

{\bf {Claim }2.} $3\leq k \leq 4$.

We first show that $k \geq 3$. If not, then $G$ is an odd cycle,
which contradicts that $G$ is of even order
($|G|=2|\overline{S}|+2$).

Next we will show $k\leq 4$. Since $C$ is minimum, $C$ has no
chords. So there are $l(k-2)$ edges between $V(C)$ and
$\overline{V(C)}=V(G)\setminus V(C)$. If $G$ is isomorphic to
$K_{4}$, then $k=3\leq 4$.  $G$ is not isomorphic to $K_{6}$.
Otherwise for any $S\subseteq V(G)$ with $|S|=|\overline{S}|+2$, we
have $|\overline{S}|=2$ and hence $G[\overline{S}]$ is not empty,
which contradicts
 the hypothesis. Hence we only need to consider that $G$ is not isomorphic to
 $K_{4}$ or $K_{6}$. By Lemma
\ref{adjacenttoatmosttwo}, any vertex in $\overline{V(C)}$ can only
be adjacent to at most two vertices in $V(C)$. So there are at least
$\frac{l(k-2)}{2}$ vertices in $\overline{V(C)}$.

By Claim 1, $|C|=l\geq \frac{|G|}{2}$. On one hand,
$|\overline{V(C)}|\geq \frac{l(k-2)}{2}\geq
\frac{|G|\times(k-2)}{4}$. On the other hand,
$|\overline{V(C)}|=|G|-l\leq \frac{|G|}{2}$. Combining these two
inequalities, that is, $\frac{|G|\times(k-2)}{4}\leq
|\overline{V(C)}| \leq \frac{|G|}{2}$, we have $k\leq 4$.


In the following, we divide the remaining proof into two cases:
 $k=4$ and $k=3$.

 {\bf{Case 1.}}  $k=4$.

 \begin{figure}[!htbp]
\begin{center}
\includegraphics[totalheight=5.8cm]{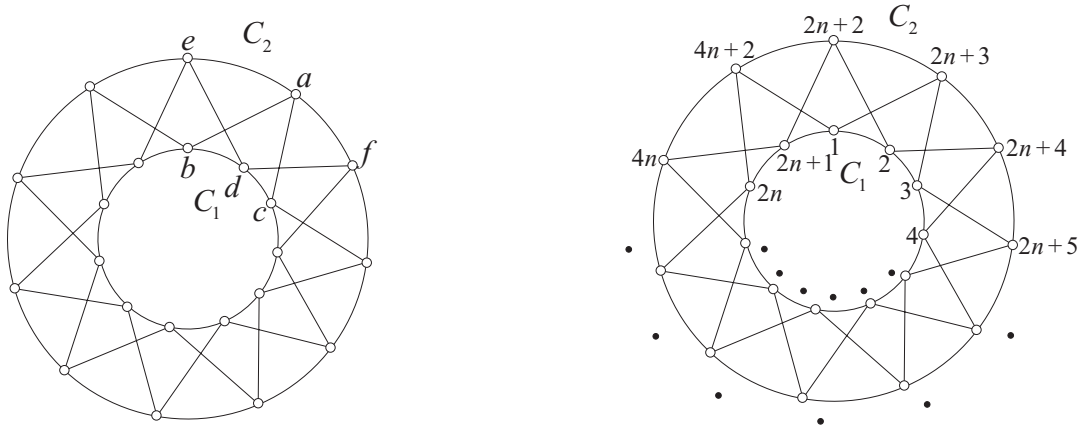}
 \caption{$Z_{4n+2}(1,4n+1,2n,
2n+2)$ and the labeling of it. \label{k4g4}}
\end{center}
\end{figure}

For any minimum odd cycle $C$, by the above proof,
$\frac{|G|}{2}\leq \frac{|G|\times(k-2)}{4}\leq |\overline{V(C)}|
\leq \frac{|G|}{2}$. Thus all equalities hold, that is,
$l=\frac{|G|}{2}$ and every vertex in $\overline{V(C)}$ is adjacent
to exactly two vertices in $V(C)$. So every vertex in
$\overline{V(C)}$ is of degree 2 in $G[\overline{V(C)}]$, the
subgraph of $G$ induced by $\overline{V(C)}$, and hence
$G[\overline{V(C)}]$ is a union of disjoint cycles. By
$|\overline{V(C)}|=|C|=l$ and $C$ is minimum, $G[\overline{V(C)}]$
is indeed a minimum odd cycle. Therefore, we conclude that the
deletion of any minimum odd cycle results in another minimum odd
cycle.


Let $C_{1}$ be a minimum odd cycle and $C_{2}$ be
 the minimum odd cycle of $G$ by deleting $V(C_{1})$ from it, we draw $C_{1}$ and $C_{2}$
 on the plane as in Figure \ref{k4g4}.
 Suppose that $a\in V(C_{2})$ is adjacent to $b$ and $c$ in $V(C_{1})$. Then by Lemma
 \ref{adjacenttoatmosttwo}, either $P_{cb}$ or $P_{bc}$ in $C_{1}$ is of length 2, we may assume that
 $P_{bc}$. The edges of $P_{cb}$ in $C_{1}$, $ca$ and $ab$ form a minimum odd cycle. By
deleting it, we obtain another minimum odd cycle. That is, $d$
should be adjacent to $e$ and $f$, where $e$ and $f$
 are neighbors of $a$ in $C_{2}$.
 By continuing this process
repeatedly, we obtain $G$ is isomorphic to the graph shown in Figure
\ref{k4g4} (left), by labeling it as shown in Figure \ref{k4g4}
(right), we can see $G$ is isomorphic to $Z_{4n+2}(1,4n+1,2n, 2n+2)$
for some integer $n$.

{\bf {Case }2.} $k=3$.

For $g(G)=3$ or $5$, Claim 1 implies that $|G|\leq 2l\leq 10$. Read
and Wilson \cite{Read98} have enumerated all connected cubic
vertex-transitive graphs on 34 and fewer vertices, from which, we
deduce that all connected non-bipartite cubic vertex-transitive
graphs of at most 10 vertices are either $Z_{4n}(1,4n-1,2n)$ or
$Z_{4n+2}(2,4n,2n+1)$ with $n= 1$ and 2, or the Petersen graph. We
are done.

For $g(G)=4$, $G$ has adjacent quadrangles. Otherwise, $|G|=4m$ for
some integer $m$. But Claim 1 and Lemma \ref{k3g4} imply
$l=\frac{|G|}{2}$ and $|G|=4t+2$ for some integer $t$, a
contradiction.  By Lemma \ref{k3g4}, $G$ is isomorphic to a graph in
$Z_{4n}(1,4n-1,2n)$ or $Z_{4n+2}(2, 4n, 2n+1)$ for some integer $n$.

{\bf{Claim 3.}} $g(G)\leq 5$. Suppose to the contrary that $g(G)\geq
6$.

There is no vertex in $\overline{V(C)}$ adjacent to two vertices in
$V(C)$. If not, then by Lemma \ref{adjacenttoatmosttwo}, we obtain a
quadrangle, a contradiction. Thus the edges sending out from $V(C)$
form a matching. That is, there are at least $l$ vertices in
$\overline{V(C)}$. By Claim 1, $l\geq \frac{|G|}{2}$, we obtain
$l=\frac{|G|}{2}$. Now, we focus on the graph
$G'=G[\overline{V(C)}]$. Since $|G'|$ is odd and each vertex in $G'$
is of degree 2, $G'$ is a union of disjoint cycles. Further, because
$l$ is the length of minimum odd cycles, $G'$ is a cycle of length
$l$.

By the above arguments, we conclude that

(i) $V(G)$ can be decomposed into two parts $V_{1}$ and $V_{2}$ such
that both $G[V_{1}]$ and $G[V_{2}]$ are minimum odd cycles.

(ii) The deletion of any minimum odd cycle from $G$ results in a
minimum cycle, too.

Since $\overline{S}$ is an independent set of $G$, by a simple
computation, we have that $G[S]$ contains three edges, and every odd
cycle
 contains either one or three edges in $G[S]$. Hence the following holds.

(iii) Every minimum odd cycle contains exactly one edge in $G[S]$.
(Otherwise, suppose there is a minimum odd cycle containing three
edges in $G[S]$. By (ii), the deletion of it results in a minimum
odd cycle which contains no edges in $G[S]$, a contradiction.)

Denote the three edges in $G[S]$ by $e_{1}, e_{2}$ and $e_{3}$. Then
by (i) and (iii), we may assume that  $C_{1}$ and $C_{2}$ are two
vertex-disjoint minimum odd cycles containing $e_{1}$ and $e_{2}$
respectively. Also by (iii), $e_{3}$ does not lie in $C_{1}$ or
$C_{2}$. For the simplicity of description, we color the vertices in
$S$ white and the vertices in $\overline{S}$ black. So two white
vertices are adjacent if and only if they are the end-vertices of
some $e_{i}$, for $i=1, 2$ or 3. There are three cases to consider.

 \begin{figure}[!htbp]
\begin{center}
\includegraphics[totalheight=4.5cm]{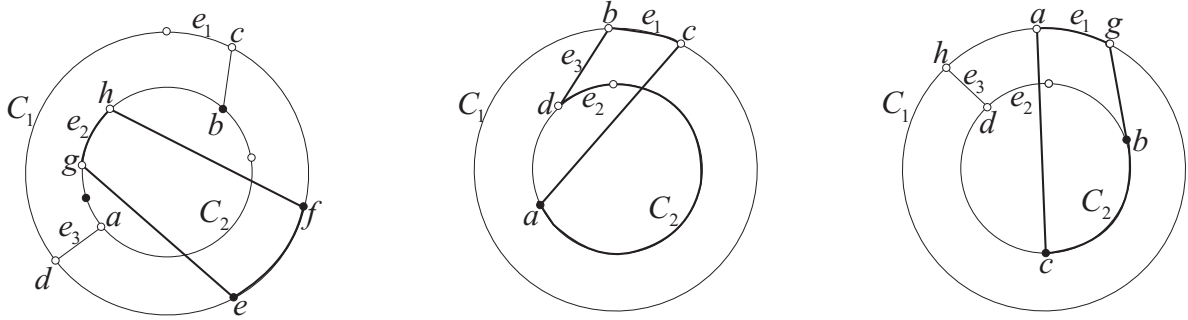}
 \caption{The illustration of proof of Case 2. \label{k3g6l72}
}
\end{center}
\end{figure}

{\bf{Subcase 2.1.}} $e_{1}$, $e_{2}$ and $e_{3}$ are independent.

We draw $C_{1}$ and $C_{2}$ on the plane as shown in Figure
\ref{k3g6l72} (left). Let $c$ be an end-vertex of $e_{1}$, $b$ the
neighbor of $c$ in $V(C_{2})$ and $a$ the end-vertex of $e_{3}$ in
$V(C_{2})$. We may assume that $e_{2}$ lies on $P_{ab}$ in $C_{2}$.
Otherwise we may redraw $C_{2}$ by interchanging $P_{ab}$ and
$P_{ba}$. The union of the edges of $P_{ab}$ in $C_{2}$, $e_{3}$,
the edges of $P_{dc}$ in $C_{1}$ and $cb$ form a cycle, denoted by
$C_{3}$. Since $C_{3}$ contains three edges in $G[S]$, it is an odd
cycle of length at least $l+2$. Simultaneously, the union of the
edges of $P_{ba}$ in $C_{2}$, $e_{3}$, the edges of $P_{cd}$ in
$C_{1}$ and $cb$ form a cycle, denoted by $C_{4}$. Since $C_{4}$
contains exactly one edge in $G[S]$, it is an odd cycle of length at
least $l$. Clearly, we have $|C_{1}|+|C_{2}|+4=|C_{3}|+|C_{4}|$.
That is, $|C_{3}|+|C_{4}|=2l+4$. Hence $l+2\leq |C_{3}|\leq l+4$ by
$|C_{4}|\geq l$.

Now we show that $C_{3}$ has no chords. Suppose by the contrary that
$C_{3}$  has a chord. Then by using the chord and $E(C_{3})$, we get
two cycles, one is of odd length denoted by $C_{5}$ and one is of
even length. By $g\geq 6$, the one with even length is of length at
least 6 and $C_{5}$ is of length at least $l$. Since the sum of the
lengths of these two cycles is at most $l+6$, we obtain that $C_{5}$
is a minimum odd cycle and the other even cycle is of length 6. By
deleting $V(C_{5})$ from $G$, we can see the resulting graph has at
least one vertex in $\{a, b, c, d\}$ of degree 3 and further cannot
be a cycle, which contradicts (ii).

Since $C_{3}$ has no chords and also $C_{2}$ has no chords, the
end-vertices of $e_{2}$ are adjacent to vertices in
$V(C_{1})\setminus V({C_{3}})$, denoted by $e$ and $f$. Then the
bold edges in Figure \ref{k3g6l72} (left) form a new cycle $C_{6}$.
Since $C_{6}$ contains only one edge in $G[S]$, it is an odd cycle.
Recall that two white vertices are adjacent if and only if they are
the end-vertices of some $e_{i}$, $P_{ef}$ in $C_{1}$ are of length
at least five. $||P_{ef}||+||P_{fe}||=l$, so for $C_{1}$, by
substituting the edges in $P_{ef}$ with the bold edges $eg$, $e_{2}$
and $hf$, we obtain that $C_{6}$ is an odd cycle of length smaller
than $l$, a contradiction.

{\bf{Subcase 2.2.}}  $e_{1}$, $e_{2}$ and $e_{3}$ form a path of
length three (see Figure \ref{k3g6l72} (middle)).

Similarly to Subcase 2.1, we may assume that $e_{1}$, $e_{2}$ and
$e_{3}$ are placed exactly like in Figure \ref{k3g6l72} (middle).
Suppose that $c$ is adjacent to a vertex $a\in V(C_{2})$.
We consider the cycle (the bold edges in Figure \ref{k3g6l72}
(middle)) consisting of the  edges of $P_{da}$ in $C_{2}$, $ac$,
$e_{1}$ and $e_{3}$. This cycle, denoted by $C_{7}$, contains
exactly three edges in $G[S]$, so it is an odd cycle and of length
at least $l+2$ by (iii). Hence $P_{da}$ in $C_{2}$ is of length at
least $l-1$ and further $da$ is an edge. Then the four edges $da,
ac, e_{1}, e_{3}$ form a quadrangle, which contradicts that
$g(G)\geq 6$.


{\bf{Subcase 2.3.}}  $e_{1}$, $e_{2}$ and $e_{3}$ form a union of an
independent edge and a path of length two (see Figure \ref{k3g6l72}
(right)).

Similarly to Subcase 2.1, we may assume that $e_{1}$, $e_{2}$ and
$e_{3}$ are placed exactly like in Figure \ref{k3g6l72} (right).
Suppose that the neighbor of $g$ in $V(C_{2})$ be $b$ and the
neighbor of $a$ in $V(C_{2})$ be $c$. Let $C_{8}$ be the cycle
formed by the edges of $P_{db}$ in $C_{2}$, $bg$, the edges of
$P_{hg}$ in $C_{1}$ and $hd$. By the similar argument as for $C_{3}$
in Subcase 2.1, we know that $C_{8}$ has no chords. Hence $c$ lies
in $P_{bd}$ in $C_{2}$. Let $C_{9}$ (the bold edges in Figure
\ref{k3g6l72} (right)) be the cycle formed by the edges of $P_{bc}$
in $C_{2}$, $ca$, $e_{1}$ and $gb$. Since $P_{cb}$ in $C_{2}$ is of
length at least three, by a similar argument as above, we get that
$C_{9}$ is a minimum odd cycle. By (ii), the deletion of $C_{9}$
results in a minimum odd cycle having no vertices of degree 3. Hence
$ha$ is an edge, but $h$ and $a$ receive the same white color, a
contradiction.
\end{proof}

\section{Matching Preclusion}

In this section, we shall prove Theorem \ref{mainresult}. We first
present the Plesn\'{i}k's Theorem which is used to estimate the
lower bound of matching preclusion number.

\begin{thm}[\cite{LovaszandPlummer86}]\label{plesnik}
If $G$ is a $k$-regular $(k-1)$-edge-connected graph of even order,
then $G-F$ has a perfect matching for every $F\subseteq E(G)$ with
$|F|\leq k-1$.
\end{thm}

For the classification of optimal solutions, Hall's Theorem (for the
bipartite case) and Tutte's Theorem (for the non-bipartite case) are
used.

\begin{thm}[Hall's Theorem \cite{Hall35}] \label{hall}
Let $G$ be a bipartite graph with bipartition $W$ and $B$.
Then $G$ has a perfect matching if and only if $|W|=|B|$ and
for any $U\subseteq W$, $|N(U)|\geq |U|$ holds.
\end{thm}


\begin{thm}[Tutte's Theorem \cite{Tutte47}] A graph $G$ has a
perfect matching if and only if $c_{o}(G-U) \leq |U|$ for any $U \subseteq V(G)$, where
$c_{o}(G-U)$ is the number of odd components of $G-U$.
\end{thm}

Now we are ready to prove Theorem \ref{mainresult}. 
Note that for the classification of optimal solutions of the
bipartite vertex-transitive graphs, the authors \cite{Cheng2012}
presented a sufficient condition for regular bipartite graphs to be
super matched with respect to the concept
``super edge-connected'', the method here is similar.\\

\noindent{\emph{\textbf{Proof of Theorem \ref{mainresult}}}}.   By
Lemma \ref{Mader71}, $G$ is $k$-edge-connected. Then by Theorem
\ref{plesnik}, $mp(G)\geq k$. Combining this with Theorem
\ref{upperbound0}, we obtain $mp(G)=k$. We are left to classify the
optimal solutions.


 Necessity. We prove by contradiction that if $G$ is isomorphic to one of the six classes, then it has
 a non-trivial optimal solution.

(a) $G$ contains a clique $S$ of size $k$ when $k$ is odd and $k\leq
|G|-2$. Since $k\leq |G|-2$, $\overline{S}$ consists of at least two
vertices. The edges between $S$ and $\overline{S}$ is a non-trivial
optimal solution.

(b) $G$ is isomorphic to a cycle of length at least six. Suppose
that $G=v_{1}v_{2}\dots v_{n}v_{1}$ with $n\geq 6$. Then the edge
set $\{v_{1}v_{2}, v_{4}v_{5}\}$ forms a non-trivial optimal
solution.

(c) $G$ is isomorphic to $Z_{4n}(1,4n-1,2n)$. Pick $F=\{12, 1(2n+1),
(2n+1)2n\}$ and $S=\{2i+1 | 1\leq i \leq n-1\}\cup \{2i | n+1\leq i
\leq 2n\}$. Then $G-F-S$ consists of  $2n+1$ isolated vertices. By
$|S|=2n-1$ and Tutte's Theorem, $G-F$ has no perfect matchings.
Hence $F$ is a non-trivial optimal solution.

(d) $G$ is isomorphic to $Z_{4n+2}(2,4n,2n+1)$. Let $F=\{(4n+1)1,
1(2n+2), (2n+2)(2n+4)\}$ and $S=\{4i-1| 1\leq i \leq n\}\cup
\{2n+2+4i | 1\leq i \leq n\}$. Similarly to (c), one can check that
$F$ is a non-trivial optimal solution.

(e) $G$ is isomorphic to  $Z_{4n+2}(1,4n+1,2n,2n+2)$. Make $F=\{12,
2(2n+2), (2n+2)(2n+3), (2n+3)1\}$ and $S=\{2i+1| 1\leq i\leq n\}\cup
\{2i | n+2\leq i \leq 2n+1\}$. Similarly to (c), one can check that
$F$ is a non-trivial optimal solution.

(f) $G$ is isomorphic to the Petersen graph. Let $F$ be the set of
three bold edges and $S$ the set of bold vertices as shown in Figure
\ref{petersengraph}. Similarly to (c), one can show $F$ is a
non-trivial optimal solution.

\begin{figure}[!htbp]
\begin{center}
\includegraphics[totalheight=4cm]{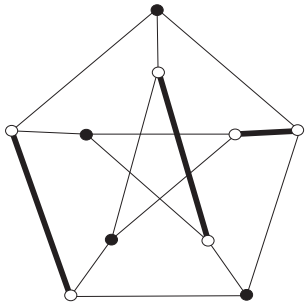}
 \caption{ The Petersen graph is not super matched. \label{petersengraph}}
\end{center}
\end{figure}

 Sufficiency. Let $F$ be an optimal solution. Then $|F|=mp(G)=k$ and $G-F$ has no
perfect matchings.  We shall show that $F$ isolates a singleton.
There are two cases to consider.

{\bf {Case }1.} $G$ is bipartite.

  We are to show that $F$ is an edge-cut.
  Assume that $W$ and $B$ are the bipartition of $G$. By Hall's theorem,
  there exists $S\subseteq W$ such that
 $|N_{G-F}(S)|\leq |S|-1$.
On the other hand, since $F$ is a matching preclusion set with the
smallest cardinality, for each edge $e\in F$, $G-F+e$ has perfect
matchings and also by Hall's Theorem, $|N_{G-F+e}(S)|\geq|S|$ holds.
Note that by adding one edge $e$ to $G-F$, the neighborhood of $S$
increases at most one vertex. Hence $|N_{G-F+e}(S)|\leq
|N_{G-F}(S)|+1$.

 Combining the above three inequations, we obtain that
 $|S|=|N_{G-F}(S)|+1$. Denote $S'=N_{G-F}(S)$.
 The edges sending out from $S$ are divided into two parts:
 One goes into $F$ and one goes into $S'$. Thus $S$ sends exactly $k|S|-|F|=k|S|-k$ edges to $S'$.
 Since $|S'|=|S|-1$, there are no edges connecting $S'$ to $W-S$.
 This implies that $F$ is an edge-cut.

If $k=1$, then $G$ is isomorphic to $K_{2}$ and $G$ is super
matched; If $G$ is a cycle of length four, then $G$ is super
matched; If $k\geq 2$, then since $G$ is bipartite, $G$ is
triangle-free, that is, $G$ is not isomorphic to a complete graph.
By hypothesis, $G$ is not isomorphic to a cycle of length at least
six. Therefore, by Theorem \ref{superlambda}, $F$ is a trivial
edge-cut, that is, $F$ isolates a singleton.

{\bf {Case 2.}}  $G$ is non-bipartite.

 By Tutte's Theorem, there exists $S\subseteq
V(G-F)$ such that $c_{o}(G-F-S)\geq |S|+2$. Now we count the number
$N$ of edges between $S$ and $\overline{S}$ in $G$. Since every
component of $G-S$ sends out at least $k$ edges, we have
$kc_{o}(G-F-S)-2k\leq N\leq k|S|,$ combining this with
$c_{o}(G-F-S)\geq |S|+2$, $k|S|\leq N\leq k|S|$ holds and further
$c_{o}(G-F-S)=|S|+2$. Hence every component sends out exactly $k$
edges, there are no even components in $G-F-S$ and each edge in $F$
connects two components in $G-F-S$.

If $|S|=0$, then there are exactly two (odd) components connected by
the edges in $F$. We claim one of them is a singleton. If not, then
by Corollary \ref{minedgecut}, $G$ contains a $k$-clique with $k$
odd and $k\leq |G|-2$, a contradiction.

If $|S|\neq0$, similarly,  by Corollary \ref{minedgecut}, each
component is a singleton. Consequently, $G$ satisfies the condition
in Lemma \ref{structure}. Hence $G\cong Z_{4n+2}(1,4n+1,2n, 2n+2)$
or $Z_{4n}(1,4n-1,2n)$ or $Z_{4n+2}(2,4n,2n+1)$ or the Petersen
graph, contradicting the hypothesis. \hspace*{4.2cm}{\large
$\square$}\\

Since the six classes of graphs in Theorem \ref{mainresult} are all
non-bipartite except for even cycles, and $Z_{4n+2}(1,4n+1,2n,
2n+2)$, $Z_{4n}(1,4n-1,2n)$, $Z_{4n+2}(2,4n,2n+1)$ and the Petersen
graph are of maximum degree at most four, the following two
corollaries arise immediately.

\begin{cor}\label{bipartite}
A connected bipartite vertex-transitive graph of even order and
other than a cycle is super matched.
\end{cor}

\begin{cor}\label{minimumdegreefive}
Let $G$ be a $k$-regular connected vertex-transitive graph of even
order and with minimum degree at least five. If it does not contain
a $k$-clique with $k$ odd and $k\leq |G|-2$, then it is super
matched.
\end{cor}
\section{Conclusion and Applications}

By the above argument, we can see that any connected
vertex-transitive graph of even order is maximally matched.
Moreover, a $k$-regular vertex-transitive graph of even order is
super matched if and only if it doesn't contain cliques of size $k$
when $k$ is odd and $k\leq |G|-2$ or it is not isomorphic to a cycle
of length at least six or $Z_{4n}(1,4n-1,2n)$ or
$Z_{4n+2}(2,4n,2n+1)$ or $Z_{4n+2}(1,4n+1,2n,2n+2)$ or the Petersen
graph. From this, the matching preclusion number and the super
matchability of the following networks with even order can be
obtained: (1) A complete graph or a complete bipartite graph; (2) a
Cayley graph generalized by transpositions or a $(n,k)$-star; (3) An
augmented cube; (4) An $(n,k)$-buddle-sort graph; (5) A tori and
related Cartesian products; (6) A burnt pancake graph; (7) A
balanced hypercube; (7) A recursive circulant $G(2^{m},4)$; (8)
$k$-ary $n$-cubes. Note that these results have been obtained in
\cite{Brigham2005}, \cite{Cheng2007}, \cite{Cheng101},
\cite{Cheng10}, \cite{Cheng20121}, \cite{Hu13}, \cite{lv2012},
\cite{Park2008} and \cite{Wang2010}, respectively and one can easily
check that the results in these papers are consistent with those
obtained by applying our results. Furthermore, we can apply our
results to other particular vertex-transitive networks, such as
folded $k$-cube graphs, Hamming graphs and halved cube graphs. We
just present the precise application for folded $k$-cubes. The
others are similar and omitted.

The folded $k$-cube graph ($k\geq 3$), containing $2^{k-1}$
vertices, denoted by $FQ_{k}$ may be formed by adding edges between
opposite pairs of vertices in a $(k-1)$-hypercube graph.
\begin{thm}
A folded $k$-cube graph is super matched if and only if  $k\geq 4$.
\end{thm}
\begin{proof}
Clearly, $FQ_{k}$ is $k$-regular. $FQ_{3}$ is $K_{4}$ and $FQ_{4}$
is the complete bipartite graph $K_{4,4}$. By the
vertex-transitivity of the $k$-hypercube graph, one can easily check
that $FQ_{k}$ is vertex-transitive. If $k=3$, $FQ_{3}=K_{4}$ is
isomorphic to $Z_{4}(1, 3, 2)$, then by Theorem \ref{mainresult},
$FQ_{3}$ is not super matched. If $k=4$, $FQ_{4}$ is a bipartite
graph other than a cycle, then by Corollary \ref{bipartite}, it is
super matched. If $k\geq 5$, then $FQ_{k}$ is of minimum degree at
least five and does not contain a $k$-clique, and by Corollary
\ref{minimumdegreefive}, it is super matched.
\end{proof}

The Hamming graph $H(d,q)$ can be viewed as the Cartesian product of
$d$ complete graphs $K_{q}$.

\begin{thm}
A Hamming graph  $H(d,q)$ with even order is super matched if and
only if $(d,q)\notin \{(1,4), (2,2)\}$.
\end{thm}

The halved $k$-cube graph or half $k$-cube graph ($k\geq 3$) is the
graph of the demihypercube, formed by connecting pairs of vertices
at distance exactly two from each other in the $k$-hypercube graph.
This connectivity pattern produces two isomorphic graphs,
disconnected from each other, each of which is the halved $k$-cube
graph.

\begin{thm}
A halved $k$-cube graph is super matched if and only if $k\geq 4$.
\end{thm}














\end{document}